\documentclass[10pt]{article}

\usepackage{amsmath}
\usepackage{amssymb}
\usepackage{amsthm}
\usepackage{graphicx}

\newcommand{\R}{\mathbb{R}}
\newcommand{\C}{\mathbb{C}}
\newcommand{\Z}{\mathbb{Z}}
\newcommand{\N}{\mathbb{N}}
\newcommand{\supp}{\mathop{\mathrm{supp}}}
\newcommand{\conv}[2]{#1\ast #2}
\newcommand{\DTFT}{\mathfrak{F}}
\newcommand{\Span}{\mathop{\mathrm{span}}}
\newcommand{\set}[2]{\left\{#1:#2\right\}}

\newtheorem{theorem}{Theorem}[section]

\newtheorem{remark}{Remark}[section]

\numberwithin{equation}{section}

\begin{document}

\title{Reproducing fractional monomials: weakening of the Strang--Fix conditions}

\author{Victor G.~Zakharov\\[0.5ex]
victor@icmm.ru}

\maketitle{}

\paragraph{Abstract}
A method to reproduce causal and symmetric monomials of fractional degree
by integer shifts of the corresponding fractional B-splines, introduced by M.~Unser and Th.~Blue, is presented.
Thus the traditional relation between the degree of reproduced monomials and
the order of approximation holds.
Bivariate, obtained by tensor product, fractional B-splines are introduced; and reproducing of bivariate
causal and symmetric monomials is shown.
Demonstration that the method is based on a weakening of the Strang--Fix conditions is presented.


\medskip



\section{Introduction}

Recall that the famous Schoenberg cardinal polynomial B-splines, see~\cite{SchoenbergA}, are formed by
peace-wise algebraic polynomials joining with the maximal continuity; and the splines
can be defined as follows
\begin{equation}\label{ClassicalBSplinesDefinition}
  B^n(x):=\int_\R B^{n-1}(x-y)B^0(y)\,dy,\quad\text{where}\quad B^0(x):=\begin{cases}1,&x\in[0,1),\\ 0,&\text{otherwise.}
                                                                                                    \end{cases}
\end{equation}

\begin{remark}
In the paper, in contrast to the traditional manner, we prefer to denote B-splines such that the order of a spline
coincides with the degree.
By the by, M.~Unser and Th.~Blue~\cite{UB1,UB2} used similar method to define the fractional B-splines.
\end{remark}

This is well known that the Fourier transform of any monomial $x^n$, $x\in\R$, $n\in\Z_{\ge0}$, is the $n$-th
derivative of the Dirac delta-distribution; and this simple relation between a monomial and
its Fourier transform cannot be extended to a fractional degree.
Nevertheless, the fashion of the Fourier transform of causal
\begin{equation}\label{SymmetricMonomial}
  x_+^\alpha:=\begin{cases}
       x^\alpha, & x\ge0,\\
       0,&\text{otherwise}
    \end{cases}\quad\text{and symmetric}\quad
  x^\alpha_*
  := \begin{cases}
     |x|^\alpha, & \alpha\not\in2\Z_{\ge0},\\
     x^\alpha \log|x|, & \alpha\in2\Z_{\ge0}
   \end{cases}
\end{equation}
monomials is the same for any {\em fractional} degree $\alpha\in\R$, $a>-1$,~\cite{GelfandShilov}.
Degree-independence of the Fourier transform form of the causal and symmetric monomials allowed
M.~Unser and Th.~Blue to extend the B-splines to non-integral orders $\alpha\in R$, $\alpha>-1$. Such
non-integral B-splines we call ``fractional''.

M.~Unser and Th.~Blue~\cite{UB1,UB2} demonstrated that the fractional (causal and symmetric) B-splines
of order $\alpha$ satisfy the Strang-Fix conditions~\cite{FixStr,StrFix} of integral order $\lceil\alpha\rceil$
and consequently reproduce algebraic polynomials up to integral degree $\lceil\alpha\rceil$.
However, Unser and Blue~\cite{UB2} proved that
a fractional B-spline of order $\alpha$ has the fractional order of approximation $\alpha+1$.

In the paper, we state a method to reproduce fractional monomials $x_+^\alpha$, $x_*^\alpha$ by integer
shifts of the corresponding fractional B-splines.
Thus we have more traditional situation, namely, for a B-spline of non-integral order $\alpha$, we have that
the degree of the reproduced monomial is fractional value $\alpha$ and the order of approximation of $B^\alpha$
is also fractional value $\alpha+1$.
Moreover, we can note that, for the causal B-splines and causal monomials,
like the classical Strang--Fix theory, we have {\em exact reproducing by a finite number}
(for any bounded interval of the ray $[0,\infty)$) of the shifted B-splines.
On the other hand, since the fractional B-splines are not compactly supported,
the reproducing of the ordinary monomials $x^n$, $n=0,1,\dots,\lceil\alpha\rceil$, by the fractional B-splines
is not exact and converges non-uniformly (only point-wise).
Note also, in spite of the infinite support of symmetric B-splines, the symmetric (fractional)
monomials are reproduced
by the symmetric B-splines {\em uniformly}.

Using tensor product, bivariate causal and symmetric (fractional, in general) B-splines are presented.
And reproducing of the corresponding (causal and symmetric) monomials is demonstrated.
Similarly to the one-dimensional case, the bivariate causal monomials are reproduced {\em exactly}
(for any bounded set of the first quadrant of the plane $\R^2$) by a
finite number of the bivariate causal B-splines.

To reproduce fractional monomials, we introduce a generalization (in fact, an weakening)
of the Strang--Fix conditions.

The paper is organized as follows. In Sec.~\ref{Section_NotationsDefinitions},
we present known, and introduced by Unser and Blue~\cite{UB1,UB2}
notations, definitions, and formulas.
In particular, the binomial series and the Fourier transforms for the causal and symmetric cases are demonstrated.
Sec.~\ref{Section_CausalSymmetricB-splines} is devoted to fractional causal and symmetric B-splines and is based on the papers~\cite{UB1,UB2}
of M.~Unser, Th.~Blue. In Subsec.~\ref{OrdinarySplines},
we concern shortly that the classical (integral) B-splines
can be considered as causal and symmetric splines. In Sec.~\ref{Section_ReproducingMonomials},
we present the method to reproduce
fractional, generally, causal and symmetric monomials by integer shifts of the corresponding B-splines.
In Sec.~\ref{Section_Strang--FixConditions}, we consider reproducing of causal and symmetric monomials
from the Strang--Fix conditions point of view.
In particular, in Subsec.~\ref{Subsec_WeakeningStrang--FixConditions},
a weakening of the Strang--Fix conditions is shown.

\section{Notations and definitions}\label{Section_NotationsDefinitions}

\subsection{General notations}

Here we introduce some well-known notations and definitions.

Let $\delta_{n0}:=\begin{cases}
     1,&n=0,\\
     0,&n\neq 0
   \end{cases}$ be the {\em Kronecker delta}  and $\delta$ be the {\em Dirac delta-distribution}.
Let $S'$ denote a space of {\em tempered distributions}.

\begin{remark}
In the paper, we shall denote numerical sequences by bold symbols and elements of the sequences
by the corresponding plain symbols:  $\pmb p:=\left(p_n\right)_{n\in\Z}$.
\end{remark}

The discrete-domain Fourier transform (DDFT) $\DTFT:\Z\to [0,2\pi)$
is defined as follows
\begin{equation*}
  \DTFT\left[\pmb p\right](\omega)
  :=\sum_{n\in\Z}p_n e^{-i n\omega}=:\hat{\pmb p}(\omega),\quad \omega\in[0,2\pi).
\end{equation*}
The inverse DDFT $\DTFT^{-1}: [0,2\pi)\to \Z$ is
defined as
\begin{equation}\label{InverseDDFT}
  p_n=\int_{[0,2\pi)}\hat{\pmb p}(\omega) e^{in\omega}\,d\omega,\qquad n\in\Z.
\end{equation}

\begin{remark}\label{RemarkOnIDDFT}
 Usually, if we have a ($2\pi$-periodic) function $\hat{\pmb p}(\omega)$
 as a polynomial (in general, the polynomial
 can have terms with negative degree) or series (the Laurent series, generally) in $e^{-i\omega}$,
 then we can take the coefficients of the polynomial (series) $\hat{\pmb p}$ as elements of the corresponding sequence
 $\pmb p$ without evaluation of the integral in formula~\eqref{InverseDDFT}.
\end{remark}

The {\em discrete} convolution $\conv{\pmb p}{\pmb q}$ of two sequences
$\pmb p:=\left(p_n\right)_{n\in\Z}$, $\pmb q:=\left(q_n\right)_{n\in\Z}$
is defined as usually
$$
  \left(\conv{\pmb p}{\pmb q}\right)_n:=\sum_{k\in \Z}p_{n-k}q_{k},\qquad n\in \Z.
$$
Recall an important  property of the DDFT:
\begin{equation}\label{DTFTForDiscreteConvolution}
  \DTFT\left[\conv{\pmb p}{\pmb q}\right](\omega)=\hat{\pmb p}(\omega)\hat{\pmb q}(\omega).
\end{equation}

By $\pmb p\cdot \pmb q$ denote the inner product:
$\pmb p\cdot \pmb q:=\sum_{k\in\Z}p_k q_k$.

The continuous Fourier transform (CFT) $F:\R\to\widehat{\R}$ of a function
$f\in L^1(\R)$ is defined as
$$
  F[f](\omega):=\int_\R f(x)e^{-i\omega x}\,dx=:\hat f(\omega).
$$

\begin{remark}\label{AnalyticallContinuation}
Note that the Fourier transform (CFT or DDFT) can be extended to compactly
supported functions (distributions) from the space
$S'(\R)$ (for DDFT, the space is $S'\bigl([0,2\pi)\bigr)$).
Note also that the Fourier transform of compactly supported, in particular, functions
can be continued analytically to the whole complex plane $\C$.
\end{remark}

The (continuum) convolution of two functions $f,g$ is defined as
$$
  \left(f\ast g\right)(x):=\int_R f(y)g(x-y)\,dy.
$$
For the CFT, an analog of formula~\eqref{DTFTForDiscreteConvolution}
is valid also
$$
  F\left[f\ast g\right](\omega)=\hat f(\omega)\hat g(\omega).
$$

\subsection{Known relations}

In this subsection, we recall some well-known and introduced by Unser and Blue~\cite{UB1,UB2} relations.

Using the Taylor series for a function $f(z):=(1+z)^\alpha$, $z\in\C$, we have a binomial series
for any (fractional, irrational, complex) exponent $\alpha$
\begin{equation}\label{NewtonBinomialSeries}
  (1+z)^\alpha=\sum_{k\ge0}\dbinom{\alpha}{k}z^k.
\end{equation}

Recall that the binomial coefficients, in the most general case, are defined as
$$
  \dbinom{m}{n}:=\dfrac{\Gamma(m+1)}{\Gamma(n+1)\Gamma(m-n+1)}.
$$

By the relation
\begin{equation*}
  \dbinom{-\alpha-1}{k}=(-1)^k\dbinom{k+\alpha}{k},
\end{equation*}
the binomial series for a negative exponent is of the form
\begin{equation}\label{NewtonBinomialSeriesNegativePower}
  (1+z)^{-\alpha-1}=\sum_{k\ge0}(-1)^k\dbinom{k+\alpha}{k}z^k.
\end{equation}

The binomial series for the symmetric case is, see~\cite{UB2},
\begin{equation}\label{NewtonBinomialSeriesSymmetric}
  |1+z|^\alpha=\sum_{k\in\Z}\dbinom{\alpha}{k+\frac a2}z^k,
\end{equation}
where the binomial coefficients
\begin{equation}\label{BinomialCoefficientForSymmetricCase}
    \dbinom{\alpha}{k+\frac a2}
\end{equation}
are even with respect to $k$.

\begin{remark}\label{RemarkOnEvenWRTKBinomialCoefficiemts}
In the case of {\em an even negative} $\alpha$,
binomial coefficients~\eqref{BinomialCoefficientForSymmetricCase}
vanish if $k\le0$, i.\,e., the binomial coefficients
are not even with respect to $k$, and we have to redefine the coefficients as
$$
  \dbinom{\alpha}{|k|+\frac a2},\qquad \alpha\in2\Z,\ \alpha<0,\ k\in\Z.
$$

On the other hand, for {\em an odd negative} $\alpha$ and any $k$, binomial
coefficients~\eqref{BinomialCoefficientForSymmetricCase} become infinite
\begin{equation}\label{BinomialCoefficientBecomeInfinity}
  \dbinom{\alpha}{k+\frac a2}=\pm\infty,\qquad \alpha\in2\Z+1,\ \alpha<0,\ k\in\Z.
\end{equation}
\end{remark}

\begin{remark}
Note, in the fractional case of $\alpha$, the binomial
series in~\eqref{NewtonBinomialSeries},~\eqref{NewtonBinomialSeriesNegativePower},
and~\eqref{NewtonBinomialSeriesSymmetric} (including any odd and excluding even $\alpha$) are infinite; and
note also that, for $|z|<1$ or for $\Re(\alpha)>0$, $|z|=1$,
the series converge absolutely (series~\eqref{NewtonBinomialSeriesSymmetric} converges only on
the unit circle, excluding $z=-1$).
\end{remark}

The CFT of the causal monomials $x_{+}^\alpha$ is of the form~\cite{GelfandShilov}
\begin{equation}\label{FourierTransformCasualMonomial}
  F\left[x_{+}^\alpha\right](\omega) =
    \begin{cases}
      \left(i\omega\right)^{-\alpha-1}, & \alpha\not\in\Z,\ \alpha>-1,\\
      \left(i\omega\right)^{-\alpha-1}+C\delta^{(\alpha)}, &
          \alpha\in\Z_{\ge0};
    \end{cases}
 \end{equation}
 where $C$ is a constant factor,
 and CFT of the symmetric monomials $x^\alpha_*$
 is~\cite{GelfandShilov}
 \begin{equation}\label{CFTSymmetricMonomials}
  F\Bigl[x^\alpha_*\Bigr](\omega) =
      \left|\omega\right|^{-\alpha-1},\qquad \alpha\in\R,\ \alpha>-1.
\end{equation}

\begin{remark}\label{RemarkAccuracyCFTMonomials}
In the paper, formulas like~\eqref{FourierTransformCasualMonomial},
\eqref{CFTSymmetricMonomials} are accurate within some constant factors.
\end{remark}

By the CFT, we shall use the simplest definition of the fractional derivative $D^\alpha$:
\begin{equation}\label{FractionalDerivative}
  D^\alpha f(x) := F^{-1}\left[(i\omega)^\alpha\hat f(\omega)\right](x).
\end{equation}
In the symmetric case, the (fractional, in general) derivative $D_*^\alpha$ is defined as follows
\begin{equation}\label{FractionalDerivativeSymmetric}
  D^\alpha_* f(x) := F^{-1}\left[|\omega|^\alpha\hat f(\omega)\right](x).
\end{equation}

\section{Causal and symmetric B-splines}\label{Section_CausalSymmetricB-splines}

In the next subsections, we consider the causal and symmetric (fractional and integer) cardinal B-splines.
And by a symbol $B^\alpha$ we shall denote B-splines of (fractional, in general) order $\alpha$ independently
of their support and symmetry.

\begin{remark}
In the paper, similarly to the Fourier transform of the causal and symmetric monomials,
see Remark~\ref{RemarkAccuracyCFTMonomials},
we define causal and symmetric B-splines within some constant factors,
i.\,e., we use no prefactors to define the splines.
\end{remark}

\subsection{Causal splines}

The causal B-splines $B_+^\alpha$, $\alpha\in\R$, $\alpha>-1$,  are defined as follows
\begin{equation}\label{FractionalB-Sline}
  B^\alpha_+(x):=\Delta^{\alpha+1}_{+}x_{+}^\alpha=\sum_{k\ge0}(-1)^k\dbinom{\alpha+1}{k}(x-k)^\alpha_+,
\end{equation}
where $\Delta^\alpha_{+}$ is a {\em forward} finite difference:
\begin{equation}\label{ForwardFinite-Difference}
  \Delta^\alpha_{+}f(x):=\sum_{k\ge0}(-1)^k\dbinom{\alpha}{k}f(x-k).
\end{equation}

Recall that the classical B-splines satisfy two-scale relations. In other words, any B-spline is a scaling
function in the corresponding multiresolution analysis (MRA). For causal B-splines, the situation is the same.
Let the sequence $\pmb a:=\left(a_k\right)_{k\in\Z_{\ge0}}$ define the mask as
\begin{equation}
    \hat{\pmb{a}}(\omega):=\DTFT[\pmb{a}](\omega)= \sum_{k=0}^\infty a_k e^{-ik\omega}.
  \label{SequencesAB}
\end{equation}
Suppose mask~\eqref{SequencesAB} has a zero of (fractional) multiplicity $\alpha+1$ at the point $\pi$.
Suppose also $\hat{\pmb{a}}(0)=1$, then the mask $\hat{\pmb{a}}$ is of the form
\begin{equation}\label{MaskA}
  \hat{\pmb{a}}(\omega)= \left(\frac{1+e^{-i\omega}}{2}\right)^{\alpha+1};
\end{equation}
and, using binomial expansion~\eqref{NewtonBinomialSeries}, we have
\begin{equation*}
  a_k := \begin{cases} \dfrac{1}{2^{\alpha+1}}\dbinom{\alpha+1}{k}, & k\ge0;\\
      0, & k<0. \end{cases}
\end{equation*}

By mask~\eqref{MaskA}
and similarly to
the known formula
for $\cos(\omega)$: $\prod_{j=1}^\infty\cos(2^{-j}\omega)=\dfrac{\sin\omega}{\omega}$, see, for example,
the book~\cite{Dau},
we can easily determine
$\hat B^\alpha_+$. Namely, we have
\begin{equation}\label{FTFractionalB-Sline}
  \hat B_+^\alpha(\omega) = \left(\frac{1-e^{-i\omega}}{i\omega}\right)^{\alpha+1}.
\end{equation}

Note that the series in~\eqref{FractionalB-Sline} is a binomial expansion
of the numerator in the right-hand side of~\eqref{FTFractionalB-Sline}; consequently
all fractional B-splines~\eqref{FractionalB-Sline} cannot be compactly supported (in fact,
$\supp B^\alpha_+=[0,\infty)$).

Using forward finite difference~\eqref{ForwardFinite-Difference},
formula~\eqref{FTFractionalB-Sline}, and
fractional derivative definition~\eqref{FractionalDerivative}, we can obtain the following relation
\begin{equation}\label{DB=DeltaB}
  D^\beta B^\alpha_+=\Delta_+^\beta B_+^{\alpha-\beta},
\end{equation}
which is valid for any $\alpha,\beta\in\R$, $\alpha>-1$, $\beta<\alpha+1$.

\subsection{Symmetric splines}

The symmetric (fractional) B-splines $B^\alpha_*$, $\alpha\in\R$, $\alpha>-1$,
are defined as
\begin{equation}\label{FractionalSymmetricB-Sline}
  B^\alpha_*(x)
    :=
         \Delta^{\alpha+1}_*x^\alpha_*=\sum\limits_{k\in\Z}(-1)^k\dbinom{\alpha+1}{k+\frac{\alpha+1}{2}}(x-k)^\alpha_*,
\end{equation}
where the symmetric monomial $x_*$ is defined by~\eqref{SymmetricMonomial} and
$\Delta^{\alpha+1}_*$ is a {\em symmetric} finite difference
\begin{equation}\label{SFDO}
  \Delta^\alpha_*f(x):=\sum_{k\in\Z}(-1)^k\dbinom{\alpha}{k+\frac{\alpha}{2}}f(x-k).
\end{equation}

Similarly to~\eqref{FTFractionalB-Sline}, the Fourier transform of the symmetric B-spline
can be determined as
\begin{equation}\label{FTSymmetricB-Sline}
  \hat B_*^\alpha(\omega)
    = \left|\frac{1-e^{-i\omega}}{\omega}\right|^{\alpha+1}.
\end{equation}

Naturally, any symmetric B-spline satisfies a two-scale relation and the DDFT of the corresponding sequence
$\pmb{a}_*:=\left(a_{*,k}\right)_{k\in\Z}$, which must be even: $\hat{\pmb a}_*(-\omega)=\hat{\pmb a}_*(\omega)$,
is of the form
\begin{equation}\label{SymmetricSequenceA}
  \hat{\pmb{a}}_*(\omega)= \left|\frac{1+e^{-i\omega}}{2}\right|^{\alpha+1},\quad\text{where}\quad
      a_{*,k} := \dfrac{1}{2^{\alpha+1}}\dbinom{\alpha+1}{k+\frac{\alpha+1}2},\quad k\in\Z.
\end{equation}

Note obvious fact that the support of any symmetric spline $B_*^\alpha$, excluding any
odd order $\alpha$, is $\R$.

Now, present an analog of formula~\eqref{DB=DeltaB}, see~\cite{UB1},
\begin{equation*}
  D^\beta_* B^\alpha_*=\Delta_*^\beta B_*^{\alpha-\beta},
\end{equation*}
where $\Delta_*^\beta$ is finite difference~\eqref{SFDO}
and the derivative $D^\beta_*$ is defined by~\eqref{FractionalDerivativeSymmetric}.

\medskip
Finally present a convolution relation, which is valid for (classical, causal, symmetric)
B-splines of any order,
$$
  B^{\alpha_1}\ast B^{\alpha_2}=B^{\alpha_1+\alpha_2+1}.
$$

\subsection{Classical B-splines}\label{OrdinarySplines}

Classical cardinal B-splines, i.\,e., the splines defined by~\eqref{ClassicalBSplinesDefinition},
can be considered as causal or symmetric. In fact, a classical B-spline $B^n$ of some order $n\in\Z_{\ge0}$
is causal if it is allocated so that $\supp B^n=[0,n+1)$.
A classical B-spline of some {\em odd} (positive) order is symmetric if
$\supp B^n=\bigl[-\lceil n/2\rceil,\lceil n/2\rceil\bigr]$, $n\in2\Z_{\ge0}+1$.

\section{Reproducing monomials}\label{Section_ReproducingMonomials}

\subsection{Ordinary polynomials}

A fractional spline $B^\alpha$ satisfies the Strang--Fix conditions of order $\lceil\alpha\rceil$.
Namely,
$$
  \hat B^\alpha(0)=1,\quad \left(\hat B^\alpha\right)^{(m)}(2\pi k),\quad m=0,\dots, \lceil\alpha\rceil, \
     k\in\Z\setminus\{0\}.
$$

Note that, unlike the traditional Strang--Fix theory, when compactly supported basis functions are considered
and the question about convergence of reproducing does not arise; in the case of fractional splines,
we have infinite series and we must trouble ourselves the convergence problem.

In paper~\cite{UB2},
there is a plot, see Figure~4.1 there,
where, in accordance with the Strang--Fix conditions, a linear polynomial is reproduced by integer shifts
of the spline $B_+^{\frac12}$:
\begin{equation}\label{MonomialReproducing}
  \sum_{k\in\Z}\left(k+\frac34\right)B_+^{\frac12}(x-k)=x,\quad x\in\R.
\end{equation}

Nevertheless, this reproducing demonstrates a problem
with the series in the left-hand side of~\eqref{MonomialReproducing}.
Namely, considering
equality~\eqref{MonomialReproducing} on the ray $[0,\infty)$,
applying derivative $D^{\frac12}$ to the both sides of the equality, and
using~\eqref{DB=DeltaB}, we get
\begin{equation}\label{ParadoxicalSituation}
  x\in[0,\infty):\quad\sum_{k\in\Z}\left(k+\frac34\right)\Delta_+^\frac12 B^0_+(x-k)=\sqrt{x},
\end{equation}
where $B^0_+=B^0$, in fact,
is the indicator function of the interval $[0,1)$.
Since the sum of the series in the left-hand side of~\eqref{ParadoxicalSituation}
is a step function, it follows that this equality cannot take place.
Thus we have to suppose that the series in the left-hand side of~\eqref{MonomialReproducing}
does not converge uniformly (only the pointwise convergence is valid); consequently,
the series cannot be differentiated (the Fourier transform cannot be applied) component-wise.
Moreover, we can demonstrate that the partition of the unity by
$B_+^{\frac12}$ also converges non-uniformly. Namely, similarly to the previous case,
we have an impossible situation
\begin{equation*}
  x\in[0,\infty):\quad\sum_{k\in\Z}\Delta_+^\frac12 B^0(x-k)=\frac1{\sqrt{x}}.
\end{equation*}

Here we do not investigate explicitly convergence of the considered series.

\subsection{Causal and symmetric monomials}

As has been demonstrated, fractional causal and symmetric B-splines are formed by linear combinations of
shifted monomials, see~\eqref{FractionalB-Sline},~\eqref{FractionalSymmetricB-Sline}.
So we can expect to obtain the causal and symmetric fractional monomials
by the corresponding causal and symmetric fractional B-splines.

\subsubsection{Causal monomials}

Let $\hat{\pmb a}(\omega):=\DTFT[\pmb a](\omega)$, $\pmb{a}:=\left(a_k\right)_{k\in\Z_{\ge0}}$,
be the mask that defines the scaling function (B-spline $B_+^\alpha$) and
be of the form~\eqref{MaskA}.
Let a detailed mask $\hat{\pmb{b}}(\omega):=\DTFT[\pmb b](\omega)$, $\pmb{b}:=\left(b_k\right)_{k\in\Z_{\ge0}}$,
which defines a wavelet,
be determined as follows
\begin{equation}\label{MaskB}
  \hat{\pmb b}(\omega):=\hat{\pmb{a}}(\omega+\pi) = \left(\frac{1-e^{-i\omega}}{2}\right)^{\alpha+1}.
\end{equation}
And, consequently, the elements of $\pmb{b}$ are
\begin{equation*}
  b_k:=(-1)^k a_k=\begin{cases}(-1)^k\dfrac{1}{2^{\alpha+1}}\dbinom{\alpha+1}{k}, & k\ge0;\\
               0, & k<0.
     \end{cases}
\end{equation*}

\begin{theorem}\label{Theorem_CasualFractionalMonomialReproducing}
Let $B^\alpha_+$ be a causal spline of order $\alpha$, where
$\alpha\in\R$, $\alpha>-1$,
is an arbitrary number.
Suppose a sequence $p$ is
\begin{equation}\label{SequenceP}
  \pmb{p}=\left(p_k\right)_{k\in\Z_{\ge0}},\qquad p_k:=
     \begin{cases}
         \dbinom{k+\alpha}{k}, & k\in\Z_{\ge0},\\
         0, &  k<0;
     \end{cases}
\end{equation}
then we have
\begin{equation}\label{CausalMonomialsRestoring}
  \sum_{k=0}^\infty p_k B_+^\alpha(x-k)=\sum_{k=0}^\infty\dbinom{k+\alpha}{k}B_+^\alpha(x-k)=x_{+}^\alpha.
\end{equation}
\end{theorem}
\begin{proof}
Consider two cases.
\begin{description}
\item[$\alpha\not\in\Z$.]
Using binomial series~\eqref{NewtonBinomialSeriesNegativePower},
the DDFT of sequence~\eqref{SequenceP} can be written as follows
\begin{equation}\label{ZTransformP}
  \hat{\pmb{p}}(\omega) :=
     \sum_{k=0}^\infty \dbinom{k+\alpha}{k}e^{-ik\omega}
      = \left(\frac{1}{1-e^{-i\omega}}\right)^{\alpha+1}.
\end{equation}
Applying the Fourier transform to the both sides of~\eqref{CausalMonomialsRestoring},
we obtain
$$
  \sum_{k=0}^\infty p_ke^{-ik\omega}\hat B_+^\alpha(\omega)=(i\omega)^{-\alpha-1}.
$$
Using~\eqref{FTFractionalB-Sline} and~\eqref{ZTransformP}, we get
\begin{equation*}
  \frac1{(1-e^{-i\omega})^{\alpha+1}}\left(\frac{1-e^{-i\omega}}{i\omega}\right)^{\alpha+1}
                 = \left(\frac{1}{i\omega}\right)^{\alpha+1}.
\end{equation*}
Thus equality~\eqref{CausalMonomialsRestoring}
is valid.
\item[$\alpha\in\Z$.]
Since the expression $(1-e^{-i\omega})^{\alpha+1}$ has a zero of
multiplicity $\alpha+1$ at the point $0$; therefore, we can rewrite the Fourier
transform of the spline $B_+^\alpha$, see~\eqref{FTFractionalB-Sline}, as follows
\begin{align*}
  \hat B_+^\alpha(\omega)=\left(\dfrac{1-e^{-i\omega}}{i\omega}\right)^{\alpha+1}
         &+ C\delta^{(\alpha)}(1-e^{-i\omega})^{\alpha+1}\\&=(1-e^{-i\omega})^{\alpha+1}
                \left((i\omega)^{-\alpha-1}+ C\delta^{(\alpha)}\right),
\end{align*}
where $C$ is a constant factor.
Using~\eqref{FourierTransformCasualMonomial} (the second line)
and~\eqref{ZTransformP}, we obtain~\eqref{CausalMonomialsRestoring}.
\end{description}
This completes the proof.
\end{proof}

In Fig.~\ref{FigureCausalMonomialSqrtxRestoring}, we present reproducing of
the causal monomials $x_+^\alpha$, $\alpha=\frac15,\frac34,1,\frac54$, by
a linear combination of integer shifts of
the causal B-splines $B_+^{\alpha}$,
see~\eqref{CausalMonomialsRestoring}.

\begin{figure}[t]
  \centerline{\includegraphics[width=0.8\textwidth]{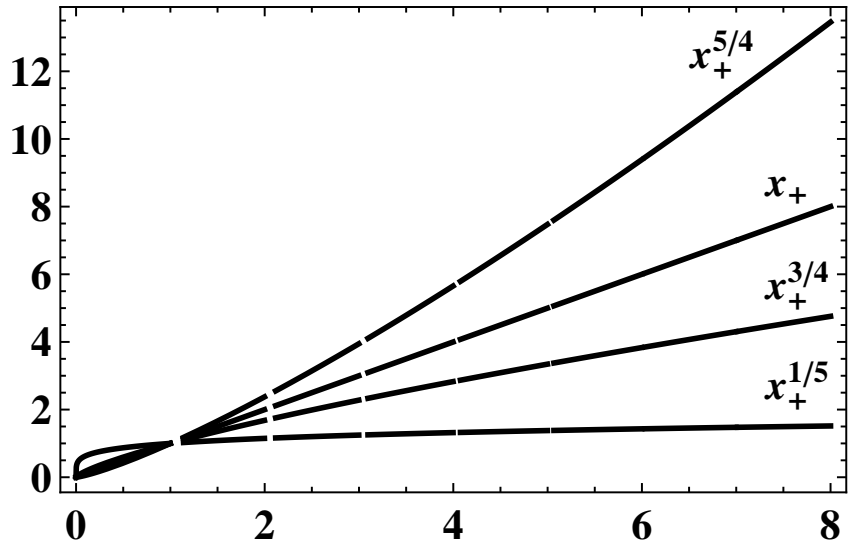}}
   \caption{Reproducing of the causal monomials $x_+^\alpha$, $x\in[0,\infty)$, $\alpha=\frac15,\frac34,1,\frac54$,
   by the causal B-splines $B_+^{\alpha}$.}
   \label{FigureCausalMonomialSqrtxRestoring}
\end{figure}

\begin{remark}
Sometimes, this is more convenient to rewrite formula~\eqref{CausalMonomialsRestoring}
as follows
\begin{equation}\label{InverseFiniteDifferenceCausalCase}
  x_{+}^\alpha=\Delta^{-\alpha-1}_+ B^\alpha_+,
\end{equation}
where the causal finite difference is defined by~\eqref{ForwardFinite-Difference}.
\end{remark}

\subsubsection{Symmetric monomials}

Let $\hat{\pmb a}_*(\omega):=\DTFT[\pmb a_*](\omega)$, $\pmb a_*:=\left(a_{*,k}\right)_{k\in\Z}$, be a mask.
Let the mask $\hat{\pmb a}_*$ be of the form~\eqref{SymmetricSequenceA}, i.\,e., the mask
has a zero of multiplicity $\alpha+1$ at the point $\pi$ and $\hat{\pmb a}_*(0)=1$.
Consequently, the mask can be used to define the B-spline $B_*^\alpha$.
Let a detailed mask $\hat{\pmb{b}}$
be determined as follows
\begin{equation}\label{MaskBSymmetric}
  \hat{\pmb b}(\omega):=\hat{\pmb{a_*}}(\omega+\pi) = \left|\frac{1-e^{-i\omega}}{2}\right|^{\alpha+1}.
\end{equation}
And, consequently, the terms of $\pmb{b}$ are
\begin{equation*}
  b_k=(-1)^k a_k=(-1)^k\dfrac{1}{2^{\alpha+1}}\dbinom{\alpha+1}{k+\frac{\alpha+1}{2}},\quad k\in\Z.
\end{equation*}

\begin{theorem}\label{Theorem_SymmetricFractionalMonomialReproducing}
Let $B^\alpha_*$ be a {\em symmetric} spline of order $\alpha$, where
$\alpha\in\R$, $\alpha>-1$,
is an arbitrary number.
Suppose a sequence $\pmb p$ is
\begin{equation}\label{SequencePSymmetricCase}
  \pmb p=\left(p_k\right)_{k\in\Z},\qquad p_k:=
     (-1)^k\dbinom{-\alpha-1}{k-\frac{\alpha+1}{2}},\quad k\in\Z,
\end{equation}
then
\begin{equation}\label{FractionalPolynomialsReproducingByFractionalSplinesSymmetricCase}
  \sum_{k\in\Z}p_k B_*^\alpha(x-k)
              =\sum_{k\in\Z}(-1)^k\dbinom{-\alpha-1}{k-\frac{\alpha+1}{2}}B_*^\alpha(x-k)=x^\alpha_*,
\end{equation}
where $x^\alpha_*$ is defined by~\eqref{SymmetricMonomial}.
\end{theorem}
The proof of this theorem is similar to the proof of Theorem~\ref{Theorem_CasualFractionalMonomialReproducing}
and based on formulas~\eqref{NewtonBinomialSeriesSymmetric},~\eqref{FTSymmetricB-Sline}.

In Fig.~\ref{FigureSymmetricMonomialSqrt|x|Restoring}, we present reproducing of
the symmetric monomials $x_*^\alpha$, $\alpha=\frac12,1,\frac32,2$, by
integer shifts of the symmetric B-splines $B_*^{\alpha}$,
see~\eqref{FractionalPolynomialsReproducingByFractionalSplinesSymmetricCase}.

\begin{figure}[t]
  \centerline{\includegraphics[width=0.8\textwidth]{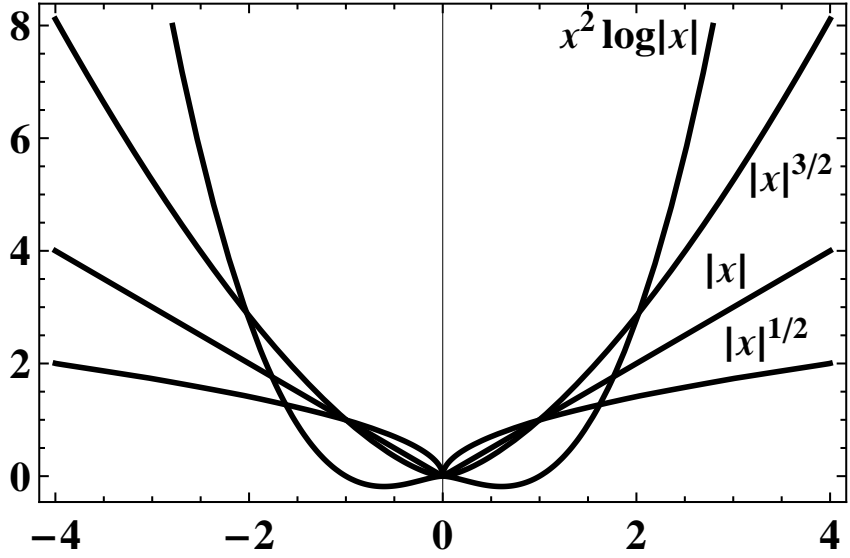}}
  \caption{Reproducing of symmetric monomials $x_*^\alpha$, $x\in\R$, $\alpha=\frac12,1,\frac32,2$, by
     the symmetric B-splines $B_*^{\alpha}$.}
   \label{FigureSymmetricMonomialSqrt|x|Restoring}
\end{figure}

\begin{remark}
In spite of the infinite support of a symmetric B-spline $B_*^{\alpha}$, $\alpha\not\in2\Z_{\ge0}+1$,
the series in the left-hand side of formula~\eqref{FractionalPolynomialsReproducingByFractionalSplinesSymmetricCase}
converges uniformly; and we can apply component-wise the fractional derivative $D_*$ of some order $<\alpha+1$
to the both sides of relation~\eqref{FractionalPolynomialsReproducingByFractionalSplinesSymmetricCase}.
Unlike reproducing of ordinary polynomials by the fractional B-splines,
formula~\eqref{FractionalPolynomialsReproducingByFractionalSplinesSymmetricCase}
can be considered as an analytical manipulation of the shifted B-splines
to obtain (fractional) monomials on the {\em whole} line $\R$ at once.
\end{remark}

In the symmetric case, the similar to~\eqref{InverseFiniteDifferenceCausalCase} formula
is valid also
$$
    x_{*}^\alpha=\Delta^{-\alpha-1}_* B^\alpha_*.
$$

\subsubsection{Causal and symmetric monomials of integral degree}

As has been noted in Subsection~\ref{OrdinarySplines},
the classical B-splines can be considered as particular cases of the causal and symmetric B-splines.
Thus formulas~\eqref{CausalMonomialsRestoring}
and~\eqref{FractionalPolynomialsReproducingByFractionalSplinesSymmetricCase}
can be applied to the classical B-splines, see
Figs.~\ref{FigureCausalMonomialSqrtxRestoring},~\ref{FigureSymmetricMonomialSqrt|x|Restoring}.

According to Remark~\ref{RemarkOnEvenWRTKBinomialCoefficiemts}, for any {\em even} degree $\alpha$ of
the reproduced monomial $x_*^\alpha$, $\alpha\in2\Z_{\ge0}$,
formula~\eqref{FractionalPolynomialsReproducingByFractionalSplinesSymmetricCase} cannot be applicable.
However this is possible to factorize a finite difference $\Delta_*^\alpha$ as follows
\begin{equation}\label{FiniteDifferenceFactorization}
  \Delta^\alpha_*=\Delta^{\alpha_1}_*\ast\Delta^{\alpha_2}_*\ast\cdots\ast
     \Delta^{\alpha_k}_*,\qquad
     \begin{aligned}
       &\alpha_1+\alpha_2+\cdots+\alpha_k=\alpha,\\
       &a_j\ne-2\N+1,\ j=1,\dots,k,
     \end{aligned}
\end{equation}
where
orders $\alpha_j$, $j=1,\dots k$,
are arbitrary (excluding negative odd values).

Using~\eqref{FiniteDifferenceFactorization}
and decomposition $3=\frac32+\frac32$, for example, the monomial $x_*^2:=x^2\log|x|$ can be reproduced as
\begin{align}
  x^2\log|x| = \Delta_*^{-3}B^2_*(x)
   =\left(\Delta^{-3/2}_*\ast\Delta^{-3/2}_*\right)&B^2_*(x)\nonumber\\
   &=\Delta^{-3/2}_*\left(\Delta^{-3/2}_*B_*^2(x)\right).
       \label{DoubleSeries}
\end{align}

Note that the series in~\eqref{DoubleSeries} converge very slowly and
this technique can be interesting from the methodical point of view only.

\subsection{Two-dimensional case}

In this subsection, we consider the simplest
two-dimensional cases of the causal and symmetric B-splines (fractional and integer) obtained by
tensor product of the one-dimensional B-splines. Namely,
\begin{equation*}
\begin{aligned}
  B^{\alpha_1,\alpha_2}_+(x,y)&:=B^{\alpha_1}_+(x)B^{\alpha_2}_+(y)&\\
     =\sum_{k_1,k_2\ge0}&(-1)^{k_1+k_2}\dbinom{\alpha_1+1}{k_1}
             \dbinom{\alpha_2+1}{k_2}(x-k_1)_+^{\alpha_1}(y-k_2)_+^{\alpha_2},\quad x\ge0,\ y\ge0;\\[0.5eX]
  B^{\alpha_1,\alpha_2}_*(x,y)&:=B^{\alpha_1}_*(x)B^{\alpha_2}_*(y)&\\
     =\sum\limits_{k_1,k_2\in\Z}&(-1)^{k_1+k_2}\dbinom{\alpha_1+1}{k_1+\frac{\alpha_1+1}{2}}
                \dbinom{\alpha_2+1}{k_2+\frac{\alpha_2+1}{2}}(x-k_1)_*^{\alpha_1}(y-k_2)_*^{\alpha_2},
                          \quad x,y\in\R.
\end{aligned}
\end{equation*}

Note that the support of any two-dimensional causal B-spline $B_+^{\alpha_1,\alpha_2}(x,y)$
(if $\alpha_1$, $\alpha_2$ are fractional) is the {\em first quadrant}
of the Cartesian plane $\R^2$
and the support of any symmetric B-spline $B_*^{\alpha_1,\alpha_2}(x,y)$ (if $\alpha_1$, $\alpha_2$ are
not odd integer) is the {\em whole plane} $\R^2$.

The Fourier transform of two-dimensional B-splines
obviously is
\begin{equation*}
\begin{aligned}
  &\hat B^{\alpha_1,\alpha_2}_+(\xi,\eta):=\hat B_+^{\alpha_1}(\xi)\hat B_+^{\alpha_2}(\eta)
     =\left(\frac{1-e^{-i\xi}}{i\xi}\right)^{\alpha_1+1}\left(\frac{1-e^{-i\eta}}{i\eta}\right)^{\alpha_2+1},\\
  &\hat B^{\alpha_1,\alpha_2}_*(\xi,\eta):=\hat B_*^{\alpha_1}(\xi)\hat B_*^{\alpha_2}(\eta)
     =\left|\frac{1-e^{-i\xi}}{\xi}\right|^{\alpha_1+1}\left|\frac{1-e^{-\eta}}{\eta}\right|^{\alpha_2+1}.
\end{aligned}
\end{equation*}

Having some B-spline $B^{\alpha_1,\alpha_2}_+$, or $B^{\alpha_1,\alpha_2}_*$, we can obtain a two-dimensional
monomial $x_+^{\alpha_1}y_+^{\alpha_2}$, or $x_*^{\alpha_1}y_*^{\alpha_2}$, respectively, as follows
\begin{align}
  &\sum_{k_1,k_2\ge0} \dbinom{k_1+\alpha_1}{k_1}\dbinom{k_2+\alpha_2}{k_2}B_+^{\alpha_1,\alpha_2}(x-k_1,y-k_2)
       =x_{+}^{\alpha_1}y_{+}^{\alpha_2};
       \label{2DFractionalPolynomialsReproducingByFractionalSplinesCausalCase}\\
  &\sum_{k_1,k_2\in\Z}(-1)^{k_1+k_2}\dbinom{-\alpha_1-1}{k_1-\frac{\alpha_1+1}{2}}
            \dbinom{-\alpha_2-1}{k_2-\frac{\alpha_2+1}{2}}B_*^{\alpha_1,\alpha_2}(x-k_1,y-k_2)
                                  =x^{\alpha_1}_*y^{\alpha_2}_*.
       \label{2DFractionalPolynomialsReproducingByFractionalSplinesSymmetricCase}
\end{align}

Similarly to the one-dimensional case, two-dimensional causal monomials are reproduced {\em exactly}
(on any bounded set of the first quadrant) by a
finite number of the bivariate causal B-splines.

Modifying a little summation in formula~\eqref{2DFractionalPolynomialsReproducingByFractionalSplinesSymmetricCase},
we can obtain that the support of the reproduced symmetric monomial $x^{\alpha_1}_*y^{\alpha_2}_*$ is
the first and third quadrants (or the second and fourth quadrants) of the plane $\R^2$. Namely, the indexes
of summation $k_1$, $k_2$ in the left-hand side of
formula~\eqref{2DFractionalPolynomialsReproducingByFractionalSplinesSymmetricCase}
must have the same (or opposite, in the second case) signs; i.\,e., $k_1\,k_2\ge0$ (or $k_1\,k_2\le0$, respectively).

In Fig.~\ref{Figures2DReproducing}, we present reproducing of 2D monomials (causal and symmetric)
by integer shifts of the corresponding 2D B-splines,
see~\eqref{2DFractionalPolynomialsReproducingByFractionalSplinesCausalCase},~\eqref{2DFractionalPolynomialsReproducingByFractionalSplinesSymmetricCase}.

\begin{figure}[h]
  \hbox to 0.8\textwidth {\large \hspace{0.18\textwidth} $\pmb{x^{1/4}_+y^{8/3}_+}$
    \hfil\hspace{0.38\textwidth} $\pmb{|x|^{1/2}|y|^{3/2}}$\hfil}

  \medskip
  \centerline{\includegraphics[width=0.45\textwidth]{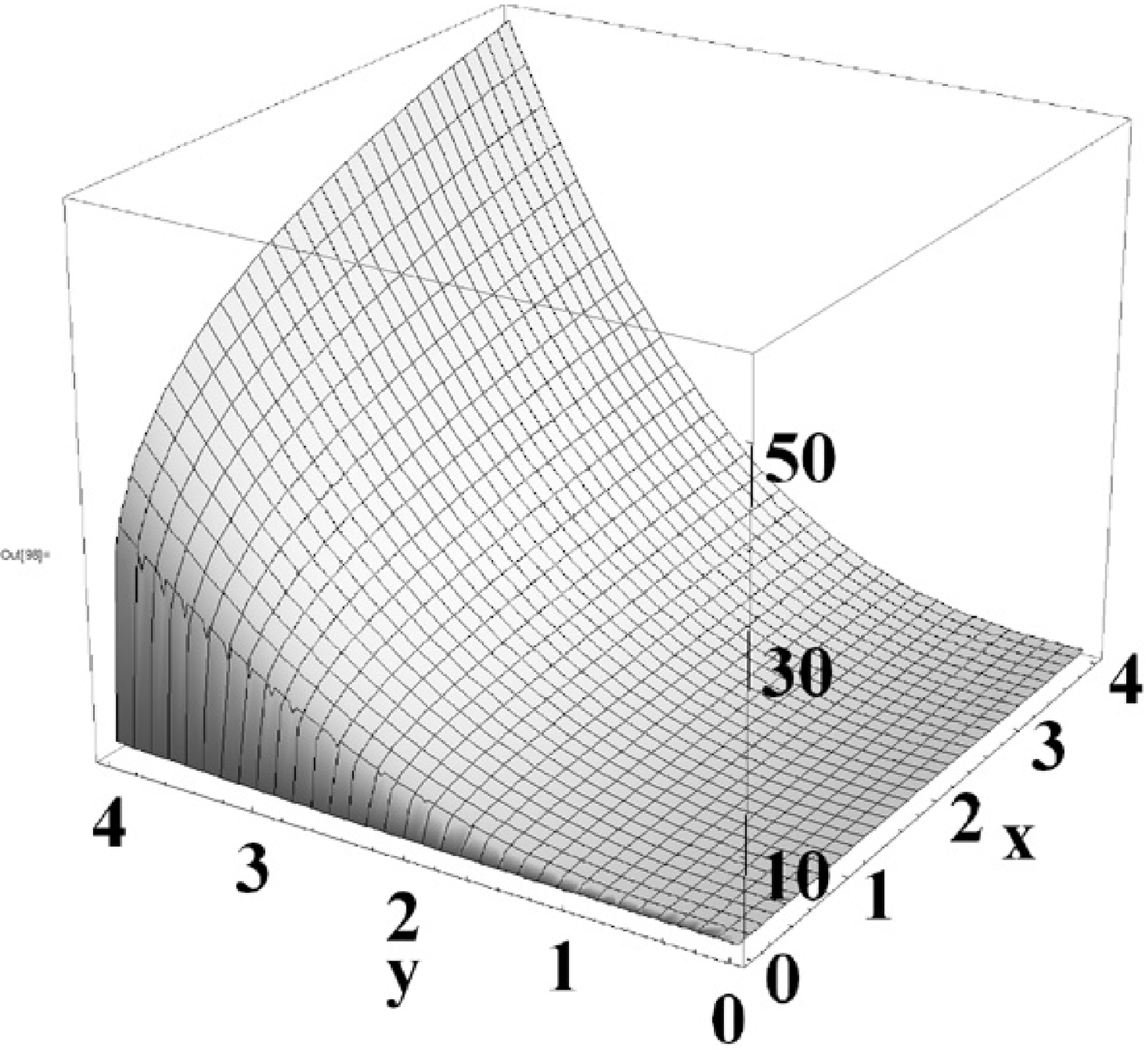}\hfil
  \includegraphics[width=0.45\textwidth]{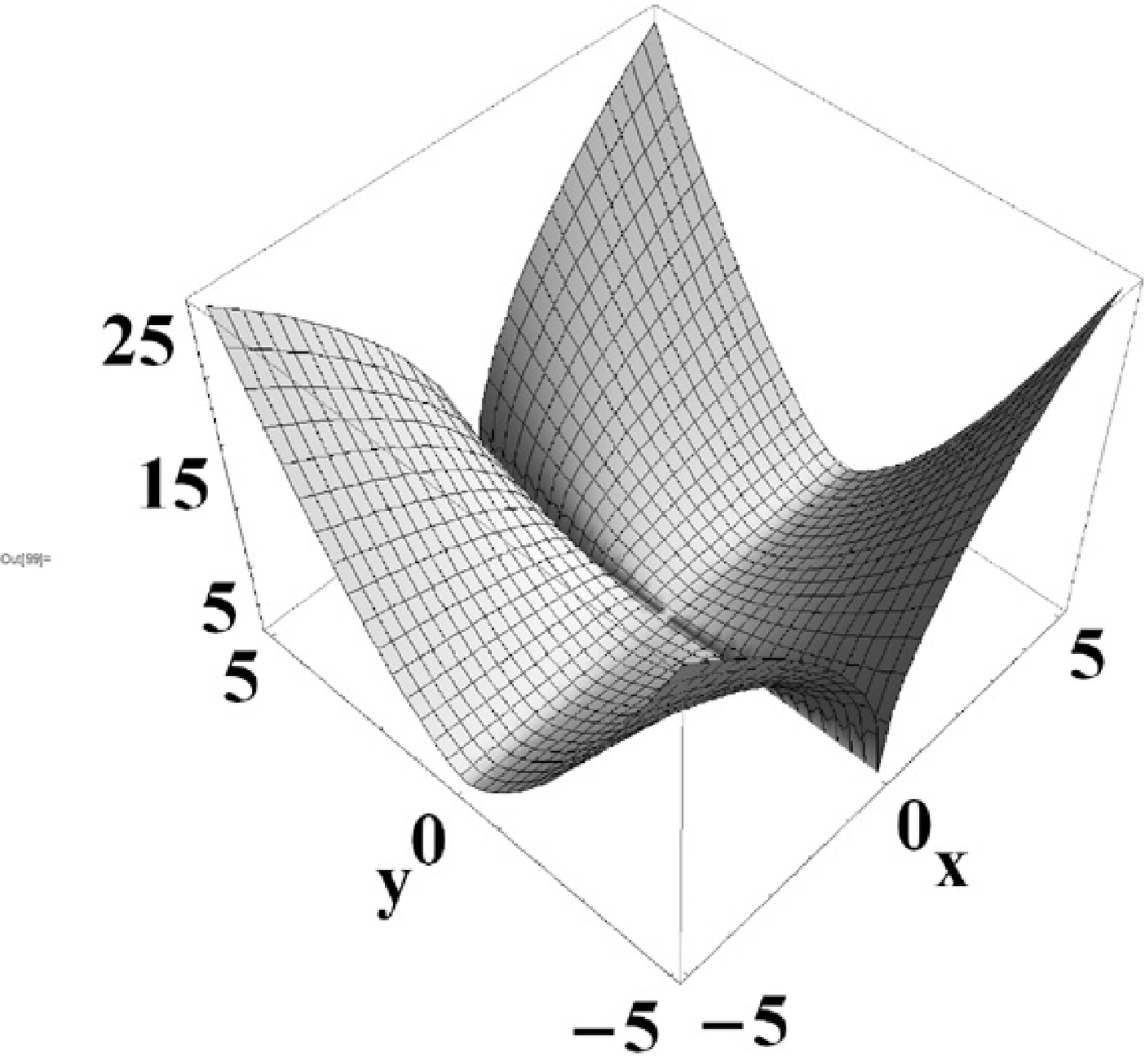}}
  \caption{2D causal monomial $x^{1/4}_+y^{8/3}_+$, $x,y\in\R^2$, $x,y\ge0$, reproduced
    by the 2D causal B-spline $B_+^{1/4,8/3}(x,y)$ (left);
    and 2D symmetric monomial $|x|^{1/2}|y|^{3/2}$, $x,y\in\R^2$, reproduced
    by the 2D symmetric B-spline $B_*^{1/2,3/2}(x,y)$ (right).}
   \label{Figures2DReproducing}
\end{figure}

\section{The Strang--Fix conditions}\label{Section_Strang--FixConditions}

In this section, we considered reproducing of the causal and symmetric monomials (integer and fractional)
with relation to the Strang--Fix conditions. The Strang--Fix conditions
are presented in the context of numerical sequences (filters) and MRA.
This section can be interpreted as an announcement and the subdivision approach to the Strang--Fix theory
will be the object of other papers, see paper~\cite{StrFix&Subdivision}.

\subsection{Classical Strang--Fix conditions}

\begin{theorem}\label{MyTheoremII}
Let $\pmb{a}=(a_k)_{k\in\Z}$ be a sequence
and
let $\pmb{b}=(b_k)_{k\in\Z}$ be another sequence defined as
\begin{equation}\label{BDefinition}
  \pmb{b}:=\DTFT^{-1}\hat{\pmb{a}}(\omega+\pi).
\end{equation}
Suppose $\hat{\pmb a}(0)=1$ and
\begin{equation}\label{DetCondition}
  \begin{vmatrix}
               \hat{\pmb a}(\omega) & \hat{\pmb a}(\omega+\pi) \\
               \hat{\pmb b}(\omega) & \hat{\pmb b}(\omega+\pi) \\
             \end{vmatrix} \ne 0, \qquad\omega\in[0,2\pi).
\end{equation}
Let $p(x)$, $x\in\R$, be a function and $\pmb p =
\bigl(p(k)\bigr)_{k\in\Z}$ be samples of the function $p$ (selected in the unit step of discretization).
Let $\phi$ be the scaling function corresponding to the
sequence $\pmb a$: $\hat\phi(\omega):=\prod_{j=1}^\infty \hat{\pmb a}(\omega/2^j)$.
Suppose
\begin{equation}\label{SequenceForm}
  \pmb b\ast{\pmb p}=0,
\end{equation}
then reproducing relations
\begin{equation*}
  \sum_{k\in\Z}p(k)\phi(x-k)=\sum_{k\in\Z}\phi(k)p(x-k)=p\ast\phi(x).
\end{equation*}
are valid.
\end{theorem}

\begin{remark}
Here note some remarks on the previous theorem.
\begin{enumerate}
\item
In fact, conditions~\eqref{SequenceForm} are equivalent to the traditional Strang--Fix conditions
on the basis function $\phi$ (defined by $\pmb a$) and the reproduced function $p$.
Certainly, conditions~\eqref{SequenceForm} do not supply directly the order of the corresponding Strang--Fix
conditions. The order is defined by the sequence $\pmb b$, actually, by the corresponding sequence $\pmb a$
(scaling function $\phi$).
\item
Definition of the sequence $\pmb b$ by formula~\eqref{BDefinition} is not necessary and, consequently,
$\pmb b$ can be defined by a different way. In any case, the sequences $\pmb a$ and $\pmb b$ must
provide~\eqref{DetCondition}.
\item
Condition~\eqref{DetCondition} is valid iff the following sum that defines the corresponding MRA
\begin{equation*}
  V_1=V_0\oplus W_0\qquad\left(\mbox{\parbox{0.27\textwidth}{the direct sum is not\\ necessary orthogonal}}\right),
\end{equation*}
where $V_0:=\Span\set{\phi(\cdot-k)}{k\in\Z}$, $V_1:=\Span\set{\sqrt{2}\phi(2\cdot-k)}{k\in\Z}$,
$\psi:=\pmb b\cdot\left(\sqrt{2}\phi(2\cdot-k)\right)_{k\in\Z}$,
and $W_0:=\Span\set{\psi(\cdot-k)}{k\in\Z}$, holds.
\item
If we rescale the function $p$ as: $p^\sharp:=p(c\cdot)$, where $c$ is a rescaling factor;
then, for the function $p^\sharp$, conditions~\eqref{SequenceForm} are valid also. Thus
$p$ can be sampled in an {\em arbitrary} step.
\end{enumerate}
\end{remark}

Using the DDFT, rewrite conditions~\eqref{SequenceForm}  as:
\begin{equation}\label{VanishingCondition}
  \hat{\pmb{b}}(\omega)\hat{\pmb{p}}(\omega)=0, \qquad\omega\in[0,2\pi).
\end{equation}
Let the sequence $\pmb{a}$ (hence, the sequence \pmb{b} also) is compactly supported, then
$\hat{\pmb{b}}(\omega)$ can vanish on a set of zero measure only. Consequently, condition~\eqref{VanishingCondition}
is valid iff the function $\hat{\pmb{p}}(\omega)$ has a zero support. There is a fact that any
distribution that is concentrated at one point is a finite sum of the delta-distribution
and its derivatives.
Thus, if the sequence $\pmb{a}$ is compactly supported, the Strang--Fix conditions can take place only
for algebraic polynomials.
If the function $\hat{\pmb{b}}(\omega)$ has a zero of multiplicity $m\in\N$ at the point $x_0\in\R$, then
a polynomial of degree no more than $m-1$
satisfies condition~\eqref{VanishingCondition}.

\begin{remark}
If we continue analytically condition~\eqref{VanishingCondition} to the whole plane $\C$, see
Remark~\ref{AnalyticallContinuation}; and
if the analytical continuation of the function $\hat{\pmb b}(\omega)$, $\omega\in\C$, has a zero
of multiplicity $m\in\N$ at a point $x_0\in\C$. Then the functions $x^n e^{x_0 x}$, $n=0,1,\dots,m-1$,
satisfy condition~\eqref{VanishingCondition} (equivalently,~\eqref{SequenceForm}) and can be reproduced
by the corresponding function $\phi$.
In the paper, we shall not consider this generalization.
Here note only the papers~\cite{Zakh20,Zakh12}.

Moreover, note that this approach can be extended to fractional {\em exponential}
splines, see paper~\cite{Massopust}.
And the monomials (multiplied by exponents) can be represented by the corresponding B-splines.
This will be discussed elsewhere.
\end{remark}

\subsection{Weakening of the Strang--Fix conditions}\label{Subsec_WeakeningStrang--FixConditions}

The Fourier transform of a causal $x_+^\alpha$ or symmetric $x^\alpha_*$ monomial,
where $\alpha\in\R$
is not necessary fractional,
cannot comply with condition~\eqref{VanishingCondition} if the function $\hat{\pmb{b}}(\omega)$ vanishes only
at a point set. Thus we have to weaken the Strang--Fix conditions~\eqref{SequenceForm}
and formulate a weakened theorem.

\begin{theorem}\label{MyTheoremIII}
Under the conditions of Theorem~\ref{MyTheoremII}, weakening the Strang--Fix
conditions~\eqref{SequenceForm} as follows
\begin{equation}\label{SequenceFormWeakened}
  \pmb b\ast{\pmb p}=\delta_{n0};
\end{equation}
we have
\begin{equation*}
  \sum_{k\in\Z}p(k)\phi(x-k)=\sum_{k\in\Z}\phi(k)p(x-k)
    =p\ast\phi(x).
\end{equation*}
\end{theorem}

Applying the DDFT to the both sides of condition~\eqref{SequenceFormWeakened}, we have
\begin{equation}\label{WeakenedStrangFixCondition}
  \hat{\pmb b}(\omega)\hat{\pmb p}(\omega)=1,\quad\text{a.\,e.}
\end{equation}
And, using~\eqref{MaskB},~\eqref{ZTransformP},~\eqref{MaskBSymmetric},~\eqref{SequencePSymmetricCase},
we see that, in
Theorems~\ref{Theorem_CasualFractionalMonomialReproducing},~\ref{Theorem_SymmetricFractionalMonomialReproducing},
condition~\eqref{WeakenedStrangFixCondition} (equivalently,~\eqref{SequenceFormWeakened}) holds, i.\,e.,
Theorems~\ref{Theorem_CasualFractionalMonomialReproducing},~\ref{Theorem_SymmetricFractionalMonomialReproducing}
can be considered as particular cases
of Theorem~\ref{MyTheoremIII}.

\begin{remark}
In Theorem~\ref{MyTheoremIII}, we do not specify the function $p=p(x)$, $x\in\R$. This is a nontrivial
problem and to define the function we need many additional explanations. Note also that even in
Theorems~\ref{Theorem_CasualFractionalMonomialReproducing},~\ref{Theorem_SymmetricFractionalMonomialReproducing},
we do not consider this function, only sequence.
\end{remark}

Condition~\eqref{SequenceFormWeakened}, in the causal case, can be written in a matrix form
that provides with the following system of linear equations
\begin{equation}\label{MatrixFormOfStrFix}
  \begin{pmatrix}
    b_0 & 0 & 0 & 0 & 0 & \dots \\
    b_1 & b_0 & 0 & 0 & 0 & \dots \\
    b_2 & b_1 & b_0 & 0 & 0 & \dots \\
    b_3 & b_2 & b_1 & b_0 & 0 & \dots \\[-1.7mm]
    \hdotsfor[1.5]{6}
  \end{pmatrix}
  \begin{pmatrix}
    p_0 \\ p_1 \\ p_2 \\ p_3 \\ p_4 \\ \vdots
  \end{pmatrix}  =
  \begin{pmatrix}
    1 \\ 0 \\ 0 \\ 0 \\ \vdots
  \end{pmatrix}.
\end{equation}

Note that system of linear equations~\eqref{MatrixFormOfStrFix}
supplies convenient matrix form of the weakened Strang--Fix conditions.
Namely, we do not bother ourselves about convergence of the binomial or Taylor series.
Moreover, the matrix form provides an useful way to state and solve the multidimensional weakened Strang--Fix
conditions. Note also that, in the multidimensional
case, several nonzero members of convolution~\eqref{SequenceFormWeakened} (on the boundary
of a compactly supported sequence $\pmb{a}$, for example) are natural.

This will be the object of another paper.

\begin{remark}
This is a well-known fact, that the space of polynomials
up to some integral degree $m$
is invariant under an {\em arbitrary} shift; and the dimension of such space is $m+1$. So the ordinary Strang--Fix
conditions provide for reproducing, by {\em integer} shifts of basis functions (B-splines, in our case), of
{\em any} polynomial (of degree up to the Strang--Fix conditions order).

In the case of causal and symmetric
(fractional and integer) monomials,
we have one-dimensional spaces, i.\,e., the spaces of causal and symmetric monomials
are not shift invariant.
Note, since
the spaces of causal and symmetric
monomials are reproduced
by integer shifts of the corresponding B-splines; the spaces
are invariant under {\em integer} shifts.
\end{remark}

\end{document}